\documentclass[a4paper,11pt]{article}

\usepackage[
    a4paper,
    left=18mm,
    right=18mm,
    top=18mm,
    bottom=18mm
]{geometry}

\usepackage[T1]{fontenc}
\usepackage{newtxtext}
\usepackage{newtxmath}
\usepackage{amsmath}

\usepackage{amsthm}
\usepackage{amsthm}

\usepackage{amssymb}
\usepackage{mathrsfs}

\usepackage{color}
\usepackage{textcomp}
\usepackage{enumitem}
\usepackage{lineno}
\usepackage{graphicx}
\usepackage{tikz}
\usepackage{comment}

\usetikzlibrary{
    positioning,
    calc,
    trees,
    decorations.markings,
    patterns,
    arrows,
    arrows.meta
}

\usepackage{diagmac2}

\usepackage[
    colorlinks=true,
    linkcolor=red,
    citecolor=cyan,
    urlcolor=blue,
    breaklinks=true
]{hyperref}

\usepackage{authblk}

\newtheorem{Theorem}{Theorem}[section]
\newtheorem{Lemma}[Theorem]{Lemma}
\newtheorem{Proposition}[Theorem]{Proposition}
\newtheorem{Corollary}[Theorem]{Corollary}
\newtheorem{Definition}[Theorem]{Definition}
\newtheorem{Example}[Theorem]{Example}
\newtheorem{Remark}[Theorem]{Remark}
\newtheorem{Conjecture}[Theorem]{Conjecture}
\newtheorem{Problem}[Theorem]{Problem}

\begin{document}

\title{On the Graph Corresponding to Center of Distances of Metric Space}

\author[1,2]{Oleksiy Dovgoshey}

\author[3]{Olga Rovenska}

\affil[1]{Department of Function Theory, Institute of Applied Mathematics and Mechanics of NASU, Cherkasy, 18031, Ukraine; oleksiy.dovgoshey@gmail.com}

\affil[2]{Department of Mathematics and  Statistics, University of Turku, Turku, 20014, Finland; oleksiy.dovgoshey@utu.fi}

\affil[3]{Department of Mathematics, Applied Mathematics and Physics, Donbas State Engineering Academy,  Kramatorsk, 84313, Ukraine; rovenskaya.olga.math@gmail.com}

\maketitle

\begin{abstract}
The center of distances $C(Y)$ of a metric space $(Y,\rho)$ is the set of all
$t\geq 0$ for which the equation
$
\rho(y,p)=t
$
has a solution for each point $p\in Y$. We say that a simple graph $CG_Y$ is the
central graph of a metric space $(Y,\rho)$ if $Y$ is the vertex set of
$CG_Y$ and distinct vertices $x,y\in Y$ are adjacent if and only if
$\rho(x,y)\in C(Y)$.
We prove the inequality
$
|E(CG_Y)| \geq \left\lceil \frac{1}{2}|Y| \right\rceil
$
for all finite metric spaces $(Y,\rho)$ with $|Y|\geq 2$
and non-empty $CG_Y.$
It is also proved that the inequality
$
|E(CG_X)| \geq |X|-1
$
holds for each finite ultrametric space $(X,d)$.
The central graphs of metric spaces $(Y,\rho)$ and ultrametric spaces
$(X,d)$ satisfying
$
|E(CG_Y)|=\left\lceil \frac{1}{2}|Y| \right\rceil
$
and, respectively,
$
|E(CG_X)|=|X|-1
$
are described up to graph isomorphism.
\end{abstract}

\hfill

{\bf Keywords:} center of distances, diametrical graph, metric space, star graph, ultrametric space.

\hfill

{\bf MSC:} Primary: 54E35; Secondary 05C05.

\section{Introduction}

The centers of distances of metric spaces was introduced in \cite{BPW2018}
as follows.

\begin{Definition}\label{oiih2} Let $(Y,\rho)$ be a metric space. The center
of distances of $(Y,\rho)$ is the subset $C(Y)\subseteq\mathbb{R}^+$,
such that  $t\in C(Y)$ if and only if the equation
\(
\rho(y,p)=t
\)
has a solution for every $p\in Y$.
\end{Definition}

Let us introduce the concept of a {\it central graph} corresponding to a given metric space.

\begin{Definition}\label{aarvg}
Let $(Y,\rho)$ be a metric space and let $C(Y)$ be the center of distances of $(Y,\rho)$. The central graph $CG_Y$ of $(Y,\rho)$ is a simple graph such that
\(
V(CG_Y)=Y
\)
and 
\[
\{u,v\}\in E(CG_Y)
\Leftrightarrow
\rho(u,v)\in C(Y)
\]
for all $u,v\in Y$. We will write that the central graph $CG_Y$ is $k$-central if $k=|C(Y)|$.
\end{Definition}

Our study was motivated by the following problem.

\begin{Problem}\label{qw3fhu}
Describe, up to graph isomorphism, the central graphs $CG_Y$ which have the smallest size
\(
|E(CG_Y)|
\)
and a given order $|V(CG_Y)|$.
\end{Problem}

The center of distances has  interesting applications to problems connected with metric spaces 
\cite{Bartoszewicz2018,Banakiewicz2022,Banakiewicz2023,Kula2025,NowakowskiPrusWisniowski2026,Bartoszewicz2024,BPW2018,DJKKR2026}.
In the case of ultrametric spaces, the structure of the center of distances was  studied in \cite{DR2026MDPI,DovgosheyRovenska2026,DJKKR2026}.

The present paper is organized as follows. Section \ref{seq2} contains some concepts and results from the theory of metric spaces and graph theory. 

The main results are formulated and proved in Section \ref{seq3}. The necessary and sufficient conditions under which the central graph is empty or complete are proved in Proposition \ref{mmhhws} and Proposition \ref{pujas}, respectively.
Theorem \ref{wsswppj} describes the structure of bicentral graphs of finite metric spaces.
In Theorem \ref{savy} we prove the sharp inequality
\[
|E(CG_X)| \geq \left\lceil \frac{|X|}{2} \right\rceil
\]
for all finite metric spaces $(X,d)$ with $|X| \geq 2$.
The structure of central graphs $CG_X$ satisfying the equality
\[
|E(CG_X)|=\left\lceil \frac{|X|}{2} \right\rceil
\]
is described in Theorem \ref{jkksw} up to graph isomorphism.
In Theorem \ref{oouy} we prove that the central graphs $CG_X$ of finite ultrametric spaces $(X,d)$ satisfy the inequality
\[
|E(CG_X)| \geq |X|-1
\]
and show that this inequality turns to equality if and only if $CG_X$ is a star graph.

In Section \ref{seq4} we formulate two conjectures and one open question related to some extremal properties of central graphs of finite metric spaces.

\section{Preliminaries}\label{seq2}

 In what follows we denote by $\mathbb{R}^{+}$ the set $[0,\infty)$.

\begin{Definition}\label{ganb}
Let $X$ be a non-empty set. A {\it metric} on  $X$ is a function $d : X \times X \to \mathbb{R}^+$,  such that for all $x,y,z \in X$:
\begin{enumerate}
\item[{\it(i)}] $d(x,y)=d(y,x)$,
\item[{\it(ii)}]  $d(x,y)=0 \iff x=y$,
\item[{\it(iii)}]  $d(x,y)\le d(x,z)+d(z,y)$.
\end{enumerate}
\end{Definition}

If the inequality
\begin{equation*}
d(x,y)\le \max \{d(x,z),d(z,y)\}
\end{equation*}
holds, then $d : X \times X \to \mathbb{R}^+$ is called an {\it ultrametric} on $X$.

Let $\mathrm{S} $ be a non-empty subset of a metric space $(X,d)$. The quantity
\[
\operatorname{diam}\mathrm{S}
:= \sup\{\, d(x,y) : x,y \in \mathrm{S} \,\}
\]
is called the {\it diameter} of $\mathrm{S}$.

In the next section of the paper we will use the concept of {\it distance set} of metric spaces.

\begin{Definition}\label{ojifrs} Let $(Y,\rho)$ be a metric space. Then the set
\(
D(Y):=\{\rho(x,y):x,y\in Y\}
\)
is called the distance set of the space $(Y,\rho)$.
\end{Definition}

Let us turn now to the graph theory.
  Recall that a {\it simple graph} is a pair $(V,E)$ consisting of a non-empty set $V$ and a (possibly empty) set $E$ whose elements are unordered pairs of distinct points from $V$.
  In what follows we suppose that all considered graphs are simple and use the term  ``graph’’ instesad of     ``simple graph’’.
For a graph $G=(V,E)$, the sets $V=V(G)$ and $E=E(G)$ are called the {\it set of vertices} and the {\it set of edges}, respectively. If $x$ and $y$ are vertices of $G$ and $\{x,y\}\in E(G)$, then we say that $x$ and $y$ are {\it adjacent} in $G$. The cardinalities of the sets $V(G)$ and $E(G)$ are called the {\it order} and the {\it size} of the graph $G$, respectively.

Recall the definition of isomorphic graphs.

\begin{Definition}\label{eddeej} Let $G_1$ and $G_2$ be graphs. A bijective
 mapping
\(
F: V(G_1) \to V(G_2)
\)
is called an isomorphism of $G_1$ and $G_2$ if the equivalence
\[
\{u,v\} \in E(G_1)
\quad\iff\quad
\{F(u),F(v)\} \in E(G_2)
\]
is valid for all $u,v \in V(G_1)$. The graphs $G_1$ and $G_2$ are
isomorphic if there exists an isomorphism of $G_1$ and $G_2$.
\end{Definition}

A graph $G$ is said to be {\it empty} if the equality $E(G)=\emptyset$ holds. 
If the membership $\{u,v\}\in E(G)$ is valid 
 for all distinct $u,v\in V(G)$, then we say
that $G$ is a {\it complete graph.}
\begin{Definition}
    \label{rre}
The {\it complement} of a graph $G$ is the graph $\overline{G}$ with
$V(\overline{G})=V(G)$, and such that
\[
\{u,v\}\in E(\overline{G})
\iff
\{u,v\}\notin E(G)
\]
for all distinct $u,v\in V(G).$
\end{Definition}

\begin{Remark}\label{iuyth} It follows directly from Definition \ref{rre} that a graph $G$ is complete if and only if the complement $\overline{G}$ is empty, and vice versa.
\end{Remark}

A graph $G$ is said to be {\it finite} if this graph contains a finite number of vertices.
A graph $H$ is a {\it subgraph} of a graph $G$, $H \subseteq G$, if 
\[
V(H) \subseteq V(G) \quad \text{and} \quad E(H) \subseteq E(G).
\] 

Let $v$ be a vertex of a graph $G$. The cardinal number of the set
\[
\{u \in V(G) : \{u,v\} \in E(G)\}
\]
is called the {\it degree} of $v$ and denoted by $\delta_G(v)$,
\begin{equation}
    \label{mmks}
\delta_G(v) := \left|\{u \in V(G) : \{u,v\} \in E(G)\}\right|.
\end{equation}

A {\it path} is a graph $P$ of the form
\begin{equation}
    \label{gasl}
V(P)=\{x_0,x_1,\dots,x_n\}, \quad E(P)=\{\{x_0,x_1\},\dots,\{x_{n-1},x_n\}\},
\end{equation}
where all $x_i$ are distinct and $n \geq 1$. If \eqref{gasl} holds, then we say that $P$ is a path joining the vertices $x_0$ and  $x_n$.  We will use the symbol $P_n$ to denote any path  containing exactly $n$ vertices.

\begin{Definition}\label{miy}
A graph $G$ is called {\it connected} if for every two distinct vertices
$u,v\in V(G)$ there exists a path
\(
P\subseteq G
\) joining the vertices $u$ and $v$.
\end{Definition}

\begin{Definition}\label{ewaf}
    \label{sgtd} The union of graphs $G_1,\ldots, G_k$ is the graph $\bigcup\limits_{i=1}^k G_i$ with vertex set $V\left( \bigcup\limits_{i=1}^k G_i\right)=\bigcup\limits_{i=1}^k V(G_i)$ and edge set $E\left(\bigcup\limits_{i=1}^k G_i\right)=\bigcup\limits_{i=1}^k E(G_i)$. The union $\bigcup\limits_{i=1}^k G_i$ is a disjoint union if the vertex sets of the graphs $G_1,\ldots, G_k$ are disjoint. A disjoint union of $k$ copies of $G$ is denoted by $kG.$
\end{Definition}

A {\it component} of a graph is a maximal connected subgraph of this graph.
\begin{Proposition}{\cite[p. 16]{Bickle2020}}\label{kuts}
Any finite disconnected graph is a disjoint union of its components. 
\end{Proposition}

A graph $C$ is called a {\it cycle} if $|V(C)|\ge 3$ and there exists an enumeration $v_1,v_2,\dots,v_n$ of its vertices such that
\[
E(C)=\left\{\{v_1,v_2\},\ldots,\{v_{n-1},v_n\},\{v_n,v_1\}\right\}.
\]
We will use the symbol $C_3$ to denote any cycle $C$ with $|V(C)|=3.$
 
A graph is called {\it acyclic} if it contains no cycle.

\begin{Definition}\label{avct}
    A connected acyclic graph is called a {\it tree}.
\end{Definition}

Let us recall now some properties of trees.

\begin{Proposition}{\cite[Theorem 1.5.2 and Corollary 1.5.3]{Diestel2005}}\label{hfdsxbj} Let $G$ be a finite connected graph. Then the following conditions are equivalent.
\begin{enumerate}
\item[(i)] $G$ is a tree.
\item[(ii)] Any two distinct vertices of $G$ are joined by unique path  $P \subseteq G$.
\item[(iii)] The equality
\[
|V(G)| = 1 + |E(G)|
\]
holds.
\end{enumerate}
\end{Proposition}

\begin{Theorem}{\cite[Chapter 1, Proposition 14]{SerreTrees1980}}\label{kiii}
     Let $G$ be a connected graph. Then there exists a tree $T \subseteq G$ such that
\(
V(T) = V(G).
\)
\end{Theorem}

A special class of trees form so-called {\it star graphs}.

\begin{Definition}\label{hihu}
A graph $G$ is a \textit{star graph} if there is a vertex $c \in V(G)$ such that $c$ and $v$ are adjacent for every
$v \in V(G) \setminus \{c\}$, but for all $u,w \in V(G)$ we have
$\{u,w\} \notin E(G)$ whenever $u \neq c$ and $w\neq c$.
\end{Definition}

The properties of ultrametric spaces generated by labeled star graphs were investigated in \cite{DCR2025OJAC,DR2025USGbLSG,DV2025JMS,DR2026PIGC}.

We will also use the concepts of {\it complete multipartite graphs} and {\it diametrical graphs.} 

\begin{Definition}
    \label{edrtaa}
Let $k\geq 2$ be an integer number.
A non-empty graph $G$ is called complete $k$-partite if its vertices can be divided into disjoint non-empty sets $X_1,\ldots,X_k$ so that there are no edges joining the vertices which belong to the same $X_i$, and, moreover, any two vertices from different $X_i$ and $X_j$ are adjacent. In this case we write
\(
G=G[X_1,\ldots,X_k].
\)
\end{Definition}

We shall say that $G$ is a \emph{complete multipartite graph} if there exists $k$ such that $G$ is complete $k$-partite.

\begin{Definition}\label{okiu}
Let $(Y,\rho)$ be a finite metric space with $|Y|\geq 2$. The diametrical graph $G_Y$ of $(Y,\rho)$ is a graph such that $V(G)=Y$ and, for all $u,v\in Y$,
\[
\{u,v\}\in E(G_{Y})
\iff
\rho(u,v)=\operatorname{diam}Y.
\]
\end{Definition}

Some applications of multipartite graphs and diametrical graphs of ultrametric spaces can be found in \cite{BDK2022TAG,ChairaDovgosheyLazaiz2022,DDP2011pNUAA}. In particular, the next theorem was proved in \cite{DDP2011pNUAA}.

\begin{Theorem}\label{kmjkusa}
Let $(Y,\rho)$ be a finite ultrametric space with $|Y|\geq 2$. Then the diametrical graph
\(
G_{Y}
\)
is complete multipartite.
\end{Theorem}

\begin{Remark}\label{oyfdd} The term diametrical graph for graphs generated by metric spaces according to Definition \ref{okiu} is not generally accepted. For example, in \cite{AlAddasiAlEzeh2008, Mulder1980, WangLuRandicLi2019}, a graph $H$ is said to be diametrical if $H$ is connected and, for every $u \in V(H)$, there is the unique $v \in V(H)$ such that
\(
d_H(u,v) \geq d_H(x,y)
\)
 for all distinct $x,y \in V(H)$, where $d_H(x,y)$ is the smallest size of the paths connecting $x$ and $y$ in $H$.
\end{Remark}

\section{Properties of central graphs}\label{seq3}

The next simple lemma directly follows from Definition \ref{oiih2}.

\begin{Lemma}\label{khan} Let $ (Y,\rho)$  be a metric space. Then the inclusion $
C(Y) \supseteq \{0\}$ holds.
\end{Lemma}

The following proposition describes some necessary and sufficient conditions under which the central graph $CG_Y$ is empty.

\begin{Proposition}\label{mmhhws}
Let $(Y,\rho)$ be a metric space. Then the following statements are equivalent.
\begin{enumerate}
    \item[(i)] The central graph $CG_Y$ is empty,
    \begin{equation}
        \label{ev1}
        E(CG_Y)=\varnothing.
        \end{equation}

    \item[(ii)] There exists a vertex $v_0\in V(CG_Y)$ such that
    \begin{equation}
        \label{ev2}
        \delta_{CG_Y}(v_0)=0. 
        \end{equation}

    \item[(iii)] The center of distances $C(Y)$ of $(Y,\rho)$ contains the unique point $0$,
    \begin{equation}
\label{ev3}
        C(Y)=\{0\}. 
        \end{equation}
\end{enumerate}
\end{Proposition}

\begin{proof}
\(
(i)\Rightarrow(ii).
\)
The implication $(i)\Rightarrow(ii)$ is evidently valid, since equality \eqref{ev1} holds if and only if we have equality \eqref{ev2} for each $v_0\in V(CG_Y)$.

\(
(ii)\Rightarrow(iii).
\)
Let statement $(ii)$ hold. To prove statement $(iii)$, we note that the
inclusion $C(Y)\supseteq \{0\}$ holds by Lemma \ref{khan}.
Thus, equality \eqref{ev3} holds if and only if
\(
t\notin C(Y)
\)
for each $t> 0$.
Suppose, on the contrary that we can find 
\(
t_0\in C(Y)\) such that $t_0>0$.
Then, by Definition \ref{oiih2}, there is $y_0\in Y$ satisfying the equality
\(
\rho(v_0,y_0)=t_0.
\)
Consequently,
\(
\{v_0,y_0\}\in E(CG_Y)
\)
is valid by Definition \ref{aarvg}.
Now using formula \eqref{mmks} with $v=v_0$ and $G=CG_Y$
we obtain the inequality
\(
\delta_{CG_Y}(v_0)\geq 1.
\)
The last inequality contradicts equality \eqref{ev2}. Equality \eqref{ev3}
follows.

\(
(iii)\Rightarrow(i).
\)
If equality \eqref{ev3} holds, then we 
have \eqref{ev1},
because
the central graph $CG_Y$ is simple 
by Definition~\ref{aarvg}, 
and,
consequently, any two adjacent in $CG_Y$ vertices are different.
\end{proof}

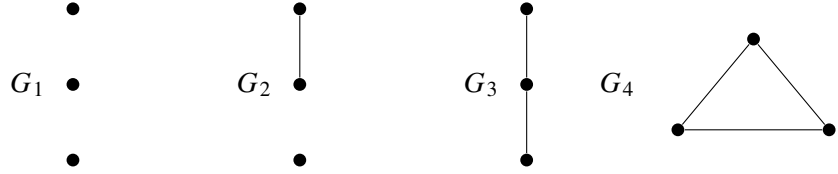
\begin{figure}[h]
    \centering
    \begin{tikzpicture}[
        vertex/.style={circle, fill=black, inner sep=1.7pt}
    ]

        \node at (-0.6,0) {$G_1$};
        \node[vertex] (g1a) at (0,1) {};
        \node[vertex] (g1b) at (0,0) {};
        \node[vertex] (g1c) at (0,-1) {};

        \node at (2.4,0) {$G_2$};
        \node[vertex] (g2a) at (3,1) {};
        \node[vertex] (g2b) at (3,0) {};
        \node[vertex] (g2c) at (3,-1) {};

        \draw (g2a) -- (g2b);

        \node at (5.4,0) {$G_3$};
        \node[vertex] (g3a) at (6,1) {};
        \node[vertex] (g3b) at (6,0) {};
        \node[vertex] (g3c) at (6,-1) {};

        \draw (g3a) -- (g3b) -- (g3c);

        \node at (7.2,0) {$G_4$};
        \node[vertex] (g4a) at (9,0.6) {};
        \node[vertex] (g4b) at (8,-0.6) {};
        \node[vertex] (g4c) at (10,-0.6) {};

        \draw (g4a) -- (g4b) -- (g4c) -- (g4a);

    \end{tikzpicture}
       \caption{Each graph $G$ of order $3$ is isomorphic to one and only one from the graphs $G_1$, $G_2$, $G_3$, $G_4$.}\label{fig1}
\end{figure}

\begin{Example}\label{piigdr}
    Proposition \ref{mmhhws} implies that the non-empty central graph $CG_Y$ of three-point metric space $(Y,\rho)$ is isomorphic either to the paths $G_3$ or to the cycle $G_4$ depicted by Figure \ref{fig1}.
\end{Example}

\begin{Proposition}\label{pujas}
Let $(Y,\rho)$ be a metric space with the distance set $D(Y)$ and the center of distances $C(Y)$. Then the following statements are equivalent.
\begin{enumerate}
    \item[(i)] The central graph $CG_Y$ is complete.
    \item[(ii)] The equality
    \begin{equation}
        \label{wr1}
        D(Y)=C(Y)
        \end{equation}
    holds.
\end{enumerate}
\end{Proposition}

\begin{proof}
$(i)\Rightarrow(ii)$. Let $CG_Y$ be a complete graph. To prove equality \eqref{wr1}, we note that the inclusion
\begin{equation}
    \label{wrwr1}
    D(Y)\supseteq C(Y)
\end{equation}
directly follows from Definition \ref{oiih2}.

Let us prove the converse inclusion
\begin{equation}
    \label{wr2}
    C(Y) \supseteq D(Y).
\end{equation}

Suppose, on the contrary, that there is
\begin{equation}
    \label{wr3}
    t_0\in D(Y)
\end{equation}
such that
\begin{equation}
\label{wr4}
    t_0\notin C(Y).
\end{equation}
Since the inclusion $C(Y)\supseteq \{0\}$ holds by Lemma \ref{khan}, relations \eqref{wr3} and \eqref{wr4} imply
\(
    t_0\in D(Y)\setminus\{0\}.
\)
Consequently, we can find two distinct points $u_0,v_0\in Y$ for which
\begin{equation}
    \label{wr5}
    \rho(u_0,v_0)=t_0>0.
\end{equation}
Definition  \ref{aarvg} and formula \eqref{wr5} imply that $u_0$ and $v_0$ are distinct non-adjacent vertices of the graph $CG_Y$, contrary to statement (i). Thus \eqref{wr2} holds. Equality \eqref{wr1} follows from inclusions \eqref{wrwr1} and \eqref{wr2}.

$(ii)\Rightarrow(i)$. The validity of the implication $(ii)\Rightarrow(i)$  follows from Definitions \ref{aarvg} and \ref{ganb}.
\end{proof}

For metric spaces with finite distance sets, Proposition \ref{pujas} admits the following
refinement.

\begin{Corollary}\label{pjaas1} Let $(Y,\rho)$ be a metric space with
finite distance set $D(Y)$. Then the following statements are equivalent.
\begin{enumerate}
    \item[(i)] The central graph $CG_Y$ is complete.
    \item[(ii)] The equality
    \(
        D(Y)=C(Y)
        \)
    holds.
    \item[(iii)] The sets $D(Y)$ and $C(Y)$ have the same
cardinality,
\(
 |D(Y)|=|C(Y)|.
\)
\end{enumerate} 
\end{Corollary}

\begin{proof} Statements (i) and (ii) are equivalent by
Proposition \ref{pujas}. It is clear that statement $(ii)$ implies statement
$(iii)$. 
The implication
\(
(iii)\Rightarrow (ii)
\)
follows from $D(Y)\supseteq C(Y)$, since for every finite set $A$
the unique subset $B\subseteq A$ satisfying the condition
\(
|B|=|A|
\)
is the set $A$ itself.
\end{proof}

\begin{figure}[ht]
    \centering
    \begin{tikzpicture}[scale=1.2,
        vertex/.style={circle, fill=black, inner sep=1.7pt}
    ]
        \node[vertex, label=below left:$y_1$]  (v1) at (0,0) {};
        \node[vertex, label=below right:$y_2$] (v2) at (3.2,0) {};
        \node[vertex, label=above right:$y_3$] (v3) at (4,2) {};
        \node[vertex, label=above left:$y_4$]  (v4) at (0.8,2) {};

        \draw (v1) -- node[below] {$3$} (v2);
        \draw (v2) -- node[right] {$4$} (v3);
        \draw (v3) -- node[above] {$3$} (v4);
        \draw (v4) -- node[left] {$4$} (v1);

        \draw (v1) -- node[pos=0.25, below right] {$5$} (v3);
        \draw (v2) -- node[pos=0.25, above right] {$5$} (v4);

        \node at (4.7,1) {$(Y,\rho)$};
    \end{tikzpicture}
    \caption{The central graph of four-point metric space $(Y,\rho)$ is complete}
    \label{fig2}
\end{figure}
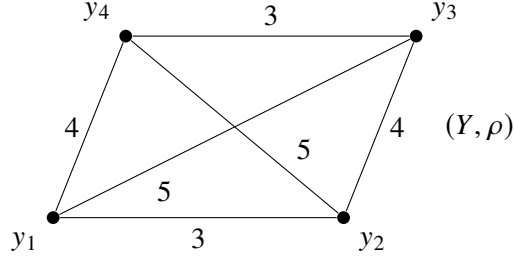

\begin{Example}\label{fghut}
Let $(Y,\rho)$ be the metric space depicted by Figure~\ref{fig2}.
Then the central graph $CG_Y$ is a complete $4$-central graph satisfying
the equality
\(
|V(CG_Y)| \leq |V(CG_X)|
\)
for any metric space $(X,d)$ whose central graph  $CG_X$ is $4$-central.
\end{Example}

\begin{Example}\label{rbhjjsa}
It was shown in \cite[Lemma 3.8]{DJKKR2026} that if $(Y,\rho)$ is a finite ultrametric space satisfying
\[
|Y|=2^{|C(Y)|-1},
\]
then
$
D(Y)=C(Y).
$
By Corollary \ref{pjaas1}, the last equality is equivalent to the completeness of the central graph $CG_Y$ and to the equality
$
|D(Y)|=|C(Y)|.
$
\end{Example}

The following theorem gives us a characterization of bicentral graphs.

\begin{Theorem}\label{wsswppj}
Let $H$ be a finite graph with the vertex set $X$ such that $|X|\geq 2$. Then the following statements are equivalent.
\begin{enumerate}
    \item[(i)] 
There exists a metric $d:X\times X\to\mathbb{R}^+$ satisfying the equalities
\begin{equation}
    \label{y1}
H=G_{X}=CG_{X},
\end{equation}
where $G_{X}$ and $CG_{X}$ are the diametrical graph of the
space $(X,d)$ and, respectively, the central graph of this
space.
    \item[(ii)] There exists a metric space $(Y,\rho)$ such that the central graph $CG_Y$ is bicentral and isomorphic to $H$.
    \item[(iii)] The inequality
    \begin{equation}
        \label{m1}
        \delta_H(u) \geq 1
        \end{equation}
    holds for every $u \in X$.
\end{enumerate}
\end{Theorem}

\begin{proof}
$(i) \Rightarrow (ii)$. 
Let $d\colon X\times X\to\mathbb{R}^+$ be a metric satisfying
equalities  \eqref{y1}.
Then, using Definition \ref{aarvg}, Definition \ref{okiu} and formula \eqref{y1}, we obtain the
equality
\begin{equation}
    \label{mmn}
C_X=\{0,\operatorname{diam} X\}. 
\end{equation}

The inequality $|X|\geq 2$ gives us the strict inequality
\(
\operatorname{diam}X>0.
\)
The last inequality and equality \eqref{mmn} imply that $CG_X$ is bicentral
by Definition \ref{aarvg}. Hence statement $(ii)$ is valid with $Y=X$ and $\rho=d$.

$(ii) \Rightarrow (iii)$. Let 
statement $(ii)$ hold. Then there is a metric space $(Y,\rho)$  with bicentral graph $CG_Y$ such that
$CG_Y$ and $H$ are isomorphic. We must prove that
inequality \eqref{m1} holds for each $u \in X$.
Suppose on the contrary  that there exists $u_0 \in X$ satisfying the equality
\begin{equation}
    \label{m2}
\delta_H(u_0)=0. 
\end{equation}
Since \(H\) and \(CG_Y\) are isomorphic, there is an isomorphism
\(
F:V(H)\to V(CG_Y).
\)
Definition \ref{eddeej} and equality \eqref{m2} give us the equality
\begin{equation}
    \label{m3}
\delta_{CG_Y}(F(u_0))=0. 
\end{equation}
The point \(F(u_0)\) belongs to the metric space \((Y,\rho)\) by Definition \ref{aarvg}. Hence \eqref{m3} implies the equality
\begin{equation}
    \label{m4}
C(Y)=\{0\}
\end{equation}
by Proposition \ref{mmhhws}. By Definition \ref{aarvg} we have the equality
\begin{equation}
    \label{m5}
|C(Y)|=2,
\end{equation}
since \(CG_Y\) is bicentral by statement $(ii)$. Equalities \eqref{m4} and \eqref{m5} cannot hold simultaneously.

Thus  equality \eqref{m2} is false. Statement $(iii)$ follows.

$(iii) \Rightarrow (i)$.
Let inequality \eqref{m1} hold for each $u\in X$. Based on the conditions of the theorem being proved, we have the inequality
\(
|X|\geq 2.
\)

Let $\overline{H}$ be the complement of the graph $H$. Since
$H$ is finite, the graph $\overline{H}$  also is finite and,
consequently, there is an injective mapping
\(
f:E(\overline{H})\to[1,2).
\)

Let us define a metric $d \colon X\times X \to \mathbb R^+$ by the rule
\begin{equation}
    \label{p3}
d(x,y):=
\begin{cases}
2, & \text{if }\{x,y\}\in E(H),\\
f(\{x,y\}), & \text{if }\{x,y\}\in E(\overline H),\\
0, & \text{otherwise}
\end{cases}
\end{equation}
for all $x,y\in X$. 
It follows directly from equality \eqref{p3} that $d$ is a symmetric function on $X\times X$, 
\begin{equation}
    \label{p4}
d(x,y)=d(y,x)
\end{equation}
for all $x,y\in X$, and the equality
\begin{equation}
    \label{p5}
    d(x,y)=0
\end{equation}
holds
if and only if $x=y$.
Moreover,
\eqref{p3} implies the double inequality
\begin{equation}
    \label{p6}
    1\leq d(x,y)\leq 2
\end{equation}
for all distinct $x,y\in X$.
Using formulas \eqref{p4}--\eqref{p6} one can easily prove that the triangle inequality
\[
d(x,y) \leq d(x,z) + d(z,y)
\]
holds for all $x,y,z \in X$.
Thus $(X,d)$ is a metric space. 

To prove the validity of statement $(i)$  it suffices to show that \eqref{y1} holds.

Let us prove the equality
\begin{equation}
    \label{sa1}
H=G_X.
\end{equation}
By conditions of the theorem being proved, we have the equality
\(
V(H)=X.
\)
Moreover, Definition \ref{okiu} implies that
\(
V(G_X)=X.
\)
Hence the equality 
\(
V(H)=V(G_X)
\)
holds.
Consequently equality \eqref{sa1}
holds if and only if
\(
E(H)=E(G_X).
\)
The inclusion
\(
f(E(\overline H))\subseteq [1,2),
\)
statement $(iii)$ and formula \eqref{p3} imply the equality
\begin{equation}
    \label{sa3}
\operatorname{diam}X=2.
\end{equation}
Now, using Definition \ref{okiu}, formula \eqref{p3} and equality \eqref{sa3} we obtain the equivalence
\[
\{x,y\}\in E(H)
\Longleftrightarrow
\{x,y\}\in E(G_X)
\]
for all $x,y\in X$. Hence we have  $E(H)=E(G_X)$.
Equality \eqref{sa1} follows.

To complete the proof of \eqref{y1}, it is sufficient to show that
\(
G_X=CG_X.
\)
The last equality holds if and only if the set
\(
\{0,2\}
\)
is the center of distances of the metric space $(X,d)$,
\begin{equation}
    \label{uy1}
C(X)=\{0,2\}.
\end{equation}

Let us prove equality \eqref{uy1}. This equality evidently holds if 
\(
|X|=2.
\)

Let us consider the case 
\begin{equation}
    \label{uy2}
    |X|\geq 3.
\end{equation}

Statement $(iii)$, formula \eqref{p3} 
and Definition \ref{aarvg} imply
the inclusion
\(
C(X)\supseteq\{0,\operatorname{diam}X\}.
\)
Consequently, \eqref{uy1} does not hold if and only if there is $t_0\in C(X)$
such that
\(
0\neq t_0\neq\operatorname{diam}X
\).

Let $0\neq t_0\neq \operatorname{diam}X$ and $t_0\in C(X)$ hold.
Then it follows from \(
0\neq t_0\neq\operatorname{diam}X
\) and formula 
\eqref{p3} that there is 
\(
\{x_0,y_0\}\in E(\overline{H})
\)
satisfying the equality
\(
f(\{x_0,y_0\})=t_0.
\)
Inequality \eqref{uy2} implies the existence $z_0\in X$ such that
\begin{equation}
    \label{wew1}
x_0\neq z_0\neq y_0.
\end{equation}

By Definition \ref{oiih2}, the relations $t_0\in C(X)$ and $t_0\neq 0$ imply that there is $x^*\in X$ such that
\(
d(x^*,z_0)=t_0.
\)
The last equality and formula \eqref{p3} give us the equality
\(
f(\{x^*,z_0\})=f(\{x_0,y_0\}).
\)
It follows from \eqref{wew1} that $\{x^*,z_0\}\neq\{x_0,y_0\}$. Hence  $f$ is non-injective, contrary to its definition. Equality \eqref{uy1} is proved. 

The proof is completed.
\end{proof}

Proposition \ref{mmhhws} and Theorem \ref{wsswppj} imply the following corollary.

\begin{Corollary}\label{waanfe} Let $X$ be a finite set and let $H$ be a graph with
$V(H)=X$. Then the following statements are equivalent.
\begin{enumerate}
    \item[(i)] There exists a metric $\rho\colon X\times X\to\mathbb{R}^{+}$
    such that $CG_X=H$ and the central graph $CG_X$ of $(X,\rho)$ is
    non-empty.

    \item[(ii)] There exists a metric $d\colon X\times X\to\mathbb{R}^{+}$
    such that $CG_X=H$ and the central graph $CG_X$ of $(X,d)$ is
    bicentral.
\end{enumerate}
\end{Corollary}

The following theorem gives the sharp lower bound of the size of finite
 non-empty central graphs having
a given order.

\begin{Theorem}\label{savy}
Let $(X,d)$ be a finite metric space containing $n\geq 2$ points.
If the central graph $CG_X$ is non-empty, then the
inequality
\begin{equation}\label{rer1}
|E(CG_X)|\geq \left\lceil\frac{n}{2}\right\rceil
\end{equation}
holds, where
\(
x\mapsto \lceil x\rceil
\)
is the ceiling function.
Moreover, for every integer $n\geq 2$ there exists a
metric space $(Y,\rho)$ such that
\(
|Y|=n
\) 
and
\begin{equation}\label{rer2}
|E(CG_Y)|=\left\lceil\frac{n}{2}\right\rceil.
\end{equation}
\end{Theorem}

\begin{proof}
Let \(X^*\) be a finite set containing \(n\geq 2\) points and let \(d^*:X^*\times X^*\to\mathbb{R}^+\) be a metric on $X^*$ such that the central graph \(CG_{X^*}\) is non-empty and
\(
\left|E(CG_{X^*})\right|\leq \left|E(CG_Y)\right|
\)
holds for every metric space \((Y,\rho)\) with \(|Y|=|X^*|\) and  non-empty \(CG_Y\). 
It follows directly from the definition of the space $(X^*,d^*)$ that, to prove inequality \eqref{rer1}, it suffices to show that 
\begin{equation}
    \label{rer3}
\left|E(CG_{X^*})\right|
=
\left\lceil\frac{n}{2}\right\rceil.
\end{equation}

Let us prove equality \eqref{rer3}.

If $X^*$ contains exactly $2$ points, then the central
graph $CG_{X^*}$ is isomorphic to the path $P_2$, because $P_2$ is
the unique non-empty graph of order $2$. Since the equalities
\[
|E(P_2)|=1=\left\lceil\frac{1}{2}\right\rceil
\]
hold, we have \eqref{rer3} for $n=2$.

If \(n=3\), then, by Theorem \ref{wsswppj},
the central graph \(CG_{X^*}\) is isomorphic either to the path \(P_3\)
or to the cycle \(C_3\) (see Example \ref{piigdr}). Since
\[
\left|E(P_3)\right|=2<3=\left|E(C_3)\right|
\]
holds,
\(CG_{X^*}\) is isomorphic to \(P_3\) by definition of $(X^*,d^*)$.
Consequently, we have the equalities
\begin{equation}
    \label{rer11}
\left|E(CG_{X^*})\right|
=
\left|E(P_3)\right|
=
2.
\end{equation}
By the definition of ceiling function we also have the equality
\(
\left\lceil\frac{3}{2}\right\rceil=2.
\)
The last equality and \eqref{rer11} imply \eqref{rer3} for \(n=3\).

Let us consider the case
\(
n>3.
\)

First, we will show that the graph \(CG_{X^*}\) is disconnected. Suppose, on the contrary, 
that \(CG_{X^*}\) is a connected graph. Then, by Theorem \ref{kiii}, there is a tree
\begin{equation}
    \label{ur1}
T \subseteq CG_{X^*}
\end{equation}
such that
\(
V(T)=X^*.
\)
By Definition \ref{avct} every tree is a connected graph and, consequently, the
inequality $|X^*|>3$ implies 
\(
\delta_T(v)\geq 1
\)
 for every \(v\in X^*\). Hence, by Theorem \ref{wsswppj}, the tree \(T\) is isomorphic
to the bicentral graph \(CG_Y\) of a metric space \((Y,\rho)\) with \(|Y|=n\).
Since every bicentral graph is non-empty, inclusion
\eqref{ur1}
and the definition of  \((X^*,d^*)\) give us the inequalities
\[
|E(T)|\leq |E(CG_{X^*})|\leq |E(T)|.
\]
Thus the equality
\(
|E(T)|=|E(CG_{X^*})|
\)
holds. The last equality, 
the equalities
\(
n=|V(CG_{X^*})|=|V(T)|,
\)
and Proposition \ref{hfdsxbj} imply
\begin{equation}
    \label{bdhjh}
|E(CG_{X^*})|=n-1.
\end{equation}

The inequality $n> 3$ implies $n-2\geq 2$ and, consequently, we can consider the paths $P_2$ and $P_{n-2}$ with
\(
V(P_2)\cap V(P_{n-2})=\varnothing.
\)
Let $H$ be the disjoint union of $P_2$ and $P_{n-2}$,
\begin{equation*}
H:=P_2\cup P_{n-2}.
\end{equation*}
Then, using Definition \ref{ewaf} and formula \eqref{gasl}, we obtain the equalities
\begin{equation}
    \label{tcr}
|V(H)|=|V(P_2)|+|V(P_{n-2})|
=2+(n-2)=n,
\end{equation}
\begin{equation}
    \label{tcd}
|E(H)|=|E(P_2)|+|E(P_{n-2})|
=(2-1)+(n-2-1)=n-2,
\end{equation}
and the inequality
\(
\delta_H(u)\geq 1
\)
for each $u\in V(H)$. The last inequality and Theorem \ref{wsswppj}
imply the existence of a metric space $(Y^\circ,\rho^\circ)$ such
that the central graph $CG_{Y^\circ}$ is bicentral and isomorphic to $H$. Consequently, equalities \eqref{tcr}, \eqref{tcd} imply $|Y^\circ|=n$ and
\[
|E(CG_{Y^\circ})|=n-2.
\]
The last equality and equality \eqref{bdhjh} give us
the strict inequality
\begin{equation}
    \label{fav}
|E(CG_{Y^\circ})|<|E(CG_{X^*})|. 
\end{equation}
Since each bicentral graph is non-empty, inequality \eqref{fav} contradicts
the definition of the metric space $(X^*,d^*)$.
Thus the central graph $CG_{X^*}$ is disconnected, as stated above. 

Let $G_1^*, G_2^*,\ldots, G_m^*$ be the components of the central graph $CG_{X^*}$.
Then, using Definition \ref{ewaf}, we obtain the equalities
\begin{equation}
    \label{bb1}
|V(CG_{X^*})|=\sum_{i=1}^{m}|V(G_{i}^*)|, 
\end{equation}
\begin{equation}
    \label{bb2}
|E(CG_{X^*})|=\sum_{i=1}^{m}|E(G_{i}^*)|,
\end{equation}
and the equality
\begin{equation}
    \label{bb3}
\delta_{G_i^*}(x_i)=\delta_{CG_{X^*}}(x_i)
\end{equation}
for each \(i\in\{1,\ldots,m\}\) and every \(x_i\in V(G_i^*)\).

Since the central graph \(CG_{X^*}\) is non-empty, Proposition \ref{mmhhws} and equality \eqref{bb3} give us 
\(
\delta_{G_i^*}(x_i)\geq 1
\)
 for each \(i\in\{1,\ldots,m\}\) and every \(x_i\in V(G_i^*)\). Hence the inequality
\begin{equation}
    \label{bb5}
|V(G_i^*)|\geq 2 
\end{equation}
holds for each \(i\in\{1,\ldots,m\}\).

As in the proof of the disconnectedness of graph $CG_{X^*}$ we can show that, for every $i\in\{1,\ldots,m\}$, the inequality
\(
|V(G_i^*)| > 3
\)
imply that $G_i^*$ is a disconnected graph. Since all components of any graph are connected and $G_1^*,\ldots,G_m^*$ are the components of $CG_{X^*}$, the inequality 
\begin{equation}
|V(G_i^*)|  \leq 3
\label{eq:case-two-0}
\end{equation}
holds for every $i\in\{1,\ldots,m\}$.
Now  inequalities \eqref{bb5} and \eqref{eq:case-two-0} imply that we have either
\(
|V(G_i^*)| = 2
\)
or
\begin{equation}
|V(G_i^*)| = 3
\label{eq:case-three}
\end{equation}
whenever $i\in\{1,\ldots,m\}$.

We assert that there is at most one index $i\in\{1,\ldots,m\}$ for which equality \eqref{eq:case-three} holds. Suppose, on the contrary, that $i_1$ and $i_2$ are two distinct indices such that
\begin{equation}
|V(G_{i_1}^*)| = 3 = |V(G_{i_2}^*)|.
\label{eq:two-three}
\end{equation}
Without losing generality, we can assume $i_1=1$ and $i_2=2$.
The connectedness
of $G_1^*$ and $G_2^*$, the equality
$V(G_1^*) \cap V(G_2^*)=\varnothing$, and equalities
\eqref{eq:two-three} with $i_1=1$ and $i_2=2$
imply
\begin{equation}
|E(G_{1}^*\cup G_{2}^*)| \geq 4.
\label{eq:union-edges}
\end{equation}

Let $3P_2$ be the disjoint union
of $3$ copies of the path $P_2$
such that
\begin{equation}
V(3P_2)=V(G_1^*\cup G_2^*).
\label{eq:1}
\end{equation}
Definition \ref{ewaf} and formula \eqref{gasl}
imply the equalities
\begin{equation}
|E(3P_2)|=3
\label{eq:2}
\end{equation}
and
\begin{equation}
\delta_{3P_2}(u)=1
\label{eq:3}
\end{equation}
for each $u\in V(3P_2)$. Let us
consider the disjoint union $G^{**}$
of $3P_2$ and  
$G_3^*,\ldots,G_m^*$
\begin{equation}
\label{u0}
G^{**}:=3P_2\cup\left(\bigcup_{i=3}^{m}G_i^*\right).
\end{equation}
Then using \eqref{bb1}, \eqref{bb2},  \eqref{eq:1},  and \eqref{eq:2} we obtain
the equalities
\begin{equation}
|V(G^{**})|
=\sum_{i=1}^{m}|V(G_i^*)|=|X^*|
=n
\label{eq:4}
\end{equation}
and
\begin{equation}
\label{eq:5}
|E(G^{**})|
=|E(3P_2)|
+\sum_{i=3}^{m}|E(G_i^*)|
=3+|E(CG_{X^*})|
-|E(G_1^*\cup G_2^*)|.
\end{equation}
Inequality \eqref{eq:union-edges}, equality \eqref{eq:2}, and equality \eqref{eq:5} imply 
\(
|E(G^{**})|
\leq |E(CG_{X^*})|-1.
\)
Hence we have the strict inequality
\begin{equation}
\label{ie1}
\left|E(G^{**})\right| < \left|E(CG_{X^*})\right|.
\end{equation}
Since the central graph $CG_{X^*}$ is non-empty, the inequality
\(
\delta_{CG_{X^*}}(u) \geq 1
\)
holds for each $u \in V(CG_{X^*})$ by Proposition~\ref{mmhhws}. Moreover, it was noted above that equality~\eqref{eq:3} holds for every vertex $u \in V(3P_2)$. Hence \eqref{u0} and 
Definition~\ref{ewaf} give us the inequality
\[
\delta_{G^{**}}(u) \geq 1
\]
for every vertex $u \in V(G^{**})$. Consequently, by Theorem \ref{wsswppj}, there is a metric space $(Y^{\circ \circ},d^{\circ \circ})$ such that the central graph $CG_{Y^{\circ\circ}}$ is bicentral and isomorphic to $G^{**}$. Since $CG_{Y^{\circ\circ}}$ and $G^{**}$ are isomorphic, inequality \eqref{ie1} implies the inequality
\[
\left|E(CG_{Y^{\circ\circ}})\right|
<
\left|E(CG_{X^*})\right|.
\]
The last inequality and equalities \eqref{eq:4} contradict the definition of the metric space $(X^*,d^*)$.

Our assertion is proved.

Let us consider an arbitrary component \(G_i^*\) of \(CG_{X^*}\).
If \(|V(G_i^*)|=2\) holds, then \(G_i^*\) is isomorphic to the
path \(P_2\), because \(P_2\) is the unique, up to isomorphism,
connected graph having order two.
Similarly, equality
\begin{equation}
    \label{rr1}
|V(G_i^*)|=3
\end{equation}
implies that \(G_i^*\) is isomorphic
either to the cycle \(C_3\) or to the path \(P_3\).
Now, reasoning in the same way as in the case \(|X^*|=n=3\) and
using equality \eqref{bb2}, we can show that \eqref{rr1} holds if and only if \(G_i^*\) and \(P_3\) are
isomorphic.

Thus, either all components $G_1^*,\ldots,G_m^*$ of the graph $CG_{X^*}$ are isomorphic to $P_2$, or exactly one of them is isomorphic to $P_3$ and the other are isomorphic to $P_2$.

In the first case, $CG_{X^*}$ is isomorphic to the graph $mP_2$ and, consequently, the equalities
\[
|X^*|=|V(CG_{X^*})|=2m
\]
and
\[
|E(CG_{X^*})|=m
\]
hold. These equalities and the equality $n=|X^*|$ give us  equality \eqref{rer3}.

In the second case, the central graph $CG_{X^*}$ is isomorphic to the disjoint union $P_3\cup(m-1)P_2$ and, consequently, we have the equalities
\begin{equation}
    \label{p11}
|X^*|=|V(CG_{X^*})|
=3+2(m-1)=2m+1
\end{equation}
and
\begin{equation}
    \label{p22}
|E(CG_{X^*})|
=2+m-1=m+1. 
\end{equation}
The equality $n=|X^*|$, equalities \eqref{p11}, \eqref{p22} and  definition of ceiling function imply
\[
\left\lceil\frac{n}{2}\right\rceil
=
\left\lceil\frac{2m+1}{2}\right\rceil
=m+1
=|E(CG_{X^*})|.
\]
Thus, equality \eqref{rer3} holds in the second case also.

Now, to complete the proof, it suffices to note that
equality \eqref{rer3} becomes equality \eqref{rer2}  if \((Y,\rho)=(X^*,d^*)\).

The proof is completed.
\end{proof}

Corollary \ref{waanfe} and Theorem \ref{savy} imply the following.

\begin{Corollary}\label{hjfery}
Let $(X,d)$ be a finite metric space containing $n\geq 2$ points. 
If the central graph $CG_X$ is bicentral, then the
inequality
\begin{equation*}
|E(CG_X)|\geq \left\lceil\frac{n}{2}\right\rceil
\end{equation*}
holds, where
\(
x\mapsto \lceil x\rceil
\)
is the ceiling function.
Moreover, for every integer $n\geq 2$ there exists a
metric space $(Y,\rho)$ such that
\(
|Y|=n
\), the central graph $CG_Y$ is bicentral, 
and
\begin{equation*}
|E(CG_Y)|=\left\lceil\frac{n}{2}\right\rceil.
\end{equation*}
\end{Corollary}

Analyzing the proof of Theorem \ref{savy}, we can obtain the next result.

\begin{Theorem}\label{jkksw}
Let \((X,d)\) be a finite metric space containing \(n\geq 2\) points.
If \(n\) is an even number, \(n=2k\), then the equality
\[
|E(CG_X)|=k
\]
holds if and only if \(CG_X\) is isomorphic to $kP_2$, where $kP_2$ is the disjoint union of
\(k\) copies of the path \(P_2\).
If \(n\) is an odd number, \(n=2k+1\), then the equality
\[
|E(CG_X)|=k+1
\]
holds if and only if \(CG_X\) is isomorphic to the disjoint union of
\(P_3\) and  \(kP_2\).
\end{Theorem}

For finite ultrametric spaces Theorem \ref{savy} admits the following modification.

\begin{Theorem}\label{oouy}
  Let $(X,d)$ be a finite ultrametric space. Then the inequality
    \begin{equation}
        |E(CG_X)| \geq |X|-1
        \label{eq:3.1}
    \end{equation}
    holds and this inequality turns to the equality
    \begin{equation}
        |E(CG_X)| = |X|-1
        \label{eq:3.2}
    \end{equation}
    if and only if  $CG_X$ is a star graph.
\end{Theorem}

\begin{proof}
The theorem being proven is obviously true if $|X|=1$.

Let us consider the case $|X| \geq 2$. 
Let $C(X)$ be the center of distances of $(X,d)$. Then the relation 
\begin{equation}
    \label{okok}
    \operatorname{diam}X \in C(X)
\end{equation}
holds by Theorem \ref{kmjkusa}. 
 Definition  \ref{aarvg}, Definition \ref{okiu} and relation \eqref{okok} imply the inclusion 
\begin{equation}
    G_X \subseteq CG_X,
    \label{eq:3.3}
\end{equation}
where $G_X$ is the diametrical graph of $(X,d)$. Using Theorem \ref{kmjkusa} and Definition \ref{okiu} again, it is easy to see that $G_X$ is connected. Consequently, by Theorem \ref{kiii}, there is a tree $T$ such that
\begin{equation}
    T \subseteq G_X
    \qquad\text{and}\qquad
    V(T)=V(G_X).
    \label{eq:3.4}
\end{equation}
Now it follows from \eqref{eq:3.3} and \eqref{eq:3.4} that $T$ is a subgraph of $CG_X$ and
\(
    V(T)=V(CG_X).
\)
Hence, we have the inequality
\begin{equation}
    |E(T)| \leq |E(CG_X)|.
    \label{eq:3.5}
\end{equation}
Proposition \ref{hfdsxbj} implies the equality
\begin{equation}
    |E(T)|=|V(T)|-1.
    \label{eq:3.6}
\end{equation}
Inequality \eqref{eq:3.1}  follows from \eqref{eq:3.5}, \eqref{eq:3.6}, and 
\(
    V(T)=V(CG_X)=X
\)
.

Let equality \eqref{eq:3.2} hold. We need to prove that  $CG_X$ is a star graph. To do this,
we first note that $CG_X$ is a connected graph, that follows from inclusion \eqref{eq:3.3}, since the diametrical graph $G_X$ is connected and
\(
V(G_X)=V(CG_X).
\)

Since $CG_X$ is a finite connected graph, equality \eqref{eq:3.2} implies that $CG_X$ is a tree by Proposition \ref{hfdsxbj}. 
Using the equivalence $(i)\Longleftrightarrow(ii)$ from Proposition \ref{hfdsxbj}, it is easy to prove that, for each tree $T$  and every connected  $G\subseteq T$ the equality
\(
V(G)=V(T)
\)
holds
 if and only if 
\(
G=T.
\)
Hence, inclusion \eqref{eq:3.3} implies 
 the equality
\begin{equation}
    \label{sf1}
G_X=CG_X. 
\end{equation}
and, consequently, the diametrical graph $G_X$ also is a tree.

It is easy to see that, that each finite tree  $T$ with $|V(T)|\geq 2$ contains a vertex $v \in V(T)$ such that  
\(
\delta_T(v)=1.
\)
Hence, there is $x_1\in X$ such that
\(
\delta_{G_X}(x_1)=1. 
\)

By Theorem \ref{kmjkusa} the diametrical graph $G_X$ is complete $k$-partite with $k\geq 2$.
Let \(X_1,\ldots,X_k\) be the parts of \(G_X\), and let \(x_1\in X_1\) satisfy the equality 
\(\delta_{G_X}(x_1)=1.\)  Then, using Definition \ref{edrtaa} and formula \eqref{mmks}
we obtain the equalities
\(
k=2
\) and \(
|X_2|=1.
\)
Let \(c\) be the unique point of \(X_2\). Then Definition \ref{edrtaa}
and Definition \ref{okiu} imply the relation
\(
\{c,x\}\in E(G_X)
\)
for each \(x\in X_1\), and the relation
\(
\{u,v\}\notin E(G_X),
\)
whenever \(u,v\in X_1\).
Thus \(G_X\) is a star graph by Definition \ref{hihu}. Hence
\(CG_X\)  also is a star graph by equality \eqref{sf1}.

Now to complete the proof we only note that equality \eqref{eq:3.2} follows from Proposition \ref{hfdsxbj}, if $CG_X$ is a star graph, because  each star graph is a tree.

The proof is completed.
\end{proof}

\section{Expected results}\label{seq4}

The authors believe that the following two conjectures are true.

\begin{Conjecture}\label{skin}
Let $X$ be a finite set. Then the following statements are equivalent.

\begin{enumerate}
    \item[(i)] There exists a metric $d\colon X\times X\to\mathbb{R}^{+}$ such that the central graph $CG_X$ is complete and $k$-central with $k=|X|$.
    
    \item[(ii)] The number of points in $X$ is $1$, $2$, or $4$.
\end{enumerate}
\end{Conjecture}

\begin{Conjecture}\label{pro}
Let $(X,d)$ be a finite ultrametric space with $|X|\geq 4$, let the central graph $CG_X$ be $3$-central, and let $G_X$ be the diametrical graph of $(X,d)$. Then the following statements are equivalent.

\begin{enumerate}
    \item[(i)] The diametrical graph $G_X$ is complete bipartite with parts $X_1$, $X_2$ satisfying the equality 
$\min\{|X_1|,|X_2|\}=2,$ and the central graph $CG_X$ is the union $G_X$ with star graphs $S_1$ and $S_2$, $CG_X=G_X\cup S_1 \cup S_2$, such that $V(S_1)=X_1$ and $V(S_2)=X_2$.
  \item[(ii)] The inequality
\(
|E(CG_X)| \leq |E(CG_Y)|
\)
holds 
for all metric spaces $(Y,\rho)$ which have $3$-central $CG_Y$ and satisfy $|Y|=|X|.$
\end{enumerate}
\end{Conjecture}

Conjecture \ref{skin} is connected with the next open question.

\begin{Problem}
\label{rfewa}
What is the smallest cardinality of the set $X$ for which there exists a metric $d\colon X\times X\to\mathbb{R}^{+}$ such that the central graph of $(X,d)$  is $k$-central?
\end{Problem}

It should be noted here that Problem \ref{rfewa} has already been solved in~\cite{DovgosheyRovenska2026} for the case of  ultrametric spaces. In particular, the main result of  \cite{DovgosheyRovenska2026}  implies the inequality
\[
|X| \geq 2^{k-1}
\]
for each ultrametric space $(X,d)$ with $k$-central $CG_X$. The structure of ultrametric spaces $(X,d)$ having $k$-central $CG_X$ and satisfying the equality
\[
|X|=2^{k-1}
\]
is described up to isometry in~\cite{DJKKR2026}.

\subsection*{Funding}
Oleksiy Dovgoshey was supported by grant 367319 from the Research Council of Finland.

\bibliographystyle{plain}
\bibliography{bib2020.07}

@article{Bartoszewicz2024,
  author = {Bartoszewicz, A. and Filipczak, M. and Horbaczewska, G. and Lindner, S. and Prus-Wi{\'s}niowski, F.},
  title = {On the operator of center of distances between the spaces of closed subsets of the real line},
  journal = {Topological Methods in Nonlinear Analysis},
  volume = {63},
  number = {2},
  year = {2024},
  pages = {413--427}
}

@article{Bartoszewicz2018,
  author  = {Bartoszewicz, A. and G{\l}{\k{a}}b, S. and Marchwicki, J.},
  title   = {Recovering purely atomic finite measure from its range},
  journal = {Journal of Mathematical Analysis and Applications},
  volume  = {467},
  number  = {2},
  year    = {2018},
  pages   = {825--841},
}

@article{Banakiewicz2023,
  author  = {Banakiewicz, M.},
  title   = {{The center of distances of central Cantor sets}},
  journal = {Results in Mathematics},
  year    = {2023},
  volume  = {78},
  number = {6},
  pages   = {234},
}

@article{Banakiewicz2022,
  author  = {Banakiewicz, M. and Bartoszewicz, A. and Filipczak, M. and Prus-Wi{\'s}niowski, F.},
  title   = {{Center of distances and central Cantor sets}},
  journal = {Results in Mathematics},
  volume  = {77},
  number  = {5},
  pages   = {196},
  year    = {2022},
  doi     = {10.1007/s00025-022-01725-1}
}

@Article{BPW2018,
  author  = {Bielas, W. and  Plewik, S. and  Walczyńska, M.},
  title   = {{On the center of distances}},
  journal = {European Journal of Mathematics},
  year    = {2018},
  volume  = {4},
  pages   = {687--698},
}

@Article{BDK2022TAG,
  author  = {Bilet, V. and Dovgoshey, O. and Kononov, Y.},
  title   = {{Ultrametrics and Complete Multipartite Graphs}},
  journal = {Theory and Applications of Graphs},
  year    = {2022},
  volume  = {9},
  number  = {1},
  pages   = {Article 8},
}

@book{Bickle2020,
  author    = {Bickle, A.},
  title     = {Fundamentals of Graph Theory},
  series    = {Pure and Applied Undergraduate Texts},
  volume    = {43},
  publisher = {American Mathematical Society},
  address   = {Providence, RI},
  year      = {2020},
  }

@Article{ChairaDovgosheyLazaiz2022,
  author  = {Chaira, K. and Dovgoshey, O. and Lazaiz, S.},
  title   = {Bipartite graphs and best proximity pairs},
  journal = {Journal of Mathematical Sciences},
  volume  = {264},
  number  = {4},
  year    = {2022},
  pages = {369--388},
}

@Book{Diestel2005,
  author    = {Diestel, R.},
  publisher = {Springer},
  title     = {{Graph Theory}},
  year      = {2005},
  address   = {Berlin},
  edition   = {{Third}},
  series    = {Graduate Texts in Mathematics},
  volume    = {173},
}

@Article{DDP2011pNUAA,
  author   = {Dordovskyi, D. and Dovgoshey, O. and Petrov, E.},
  title    = {Diameter and diametrical pairs of points in ultrametric spaces},
  journal  = {$p$-Adic Numbers, Ultrametric Analysis and Applications},
  year     = {2011},
  volume   = {3},
  number   = {4},
  pages    = {253--262},
  fjournal = {p-Adic Numbers, Ultrametric Analysis, and Applications},
}

@article{DJKKR2026,
  author  = {Dovgoshey, O. and Jaiswal, N. S. and
             Kainth, S. P. S. and Kaur, G. and
             Rovenska, O.},
  title   = {Maximal Center of Distances of Finite Ultrametric Spaces
             and Perfect Binary Trees},
  journal = { ArXiv:2609.11384},
  year    = {2026},
}

@Article{DCR2025OJAC,
  author  = {Dovgoshey, O. and Cantor, O. and Rovenska, O.},
  title   = {Compact ultrametric spaces generated by labeled star graphs},
  journal = {Online Journal of Analytic Combinatorics},
  year    = {2025},
  volume  =  {20},
  pages   = {1--20},
}

@article{AlAddasiAlEzeh2008,
  author  = {Salah Al-Addasi and Hasan Al-Ezeh},
  title   = {Bipartite Diametrical Graphs of Diameter 4 and Extreme Orders},
  journal = {International Journal of Mathematics and Mathematical Sciences},
  volume  = {2008},
  pages   = {1--11},
  year    = {2008},
  doi     = {10.1155/2008/468583}
}

@article{Mulder1980,
  author  = {H. M. Mulder},
  title   = {$n$-Cubes and Median Graphs},
  journal = {Journal of Graph Theory},
  volume  = {4},
  number  = {1},
  pages   = {107--110},
  year    = {1980},
  doi     = {10.1002/jgt.3190040112}
}

@article{WangLuRandicLi2019,
  author  = {Jianfeng Wang and Lu Lu and Milan Randi{\'c} and Guozheng Li},
  title   = {Graph Energy Based on the Eccentricity Matrix},
  journal = {Discrete Mathematics},
  volume  = {342},
  number  = {9},
  pages   = {2636--2646},
  year    = {2019},
  doi     = {10.1016/j.disc.2019.05.033}
}

@Article{DR2025USGbLSG,
  author  = {Dovgoshey, O. and Rovenska, O.},
  journal = {Journal of Mathematical Sciences},
  title   = {Ultrametric spaces generated by labeled star graphs},
  year     = {2025},
  volume   = {228},
  number   = {2},
  pages    = {182--198},
}

@Article{DR2026MDPI,
  author       = {Dovgoshey, O. and Rovenska, O.},
  title        = {Center of distances of ultrametric spaces generated by labeled trees},
  year         = {2026},
  volume = {14},
  number = {5},
  journal= {Mathematics},
}

@article{DovgosheyRovenska2026,
  author  = {Dovgoshey, O. and Rovenska, O.},
  title   = {On the center of distances of finite ultrametric spaces},
  journal = {ArXiv:2603.25850},
  year    = {2026},
  url     = {https://arxiv.org/abs/2603.25850}
}

@article{DR2026PIGC,
  author  = {Dovgoshey, O. and Rovenska, O.},
  title   = {Forbidden four cycle, star graphs and isometric embeddings},
  journal = {Proceedings of the International Geometry Center},
  year    = {2026},
  volume  = {18},
  number  = {3},
  pages   = {306--331},
  doi     = {10.15673/pigc.v18i3.3268}
}

@article{DV2025JMS,
  author    = {O. Dovgoshey and V. Vito},
  title     = {Totally bounded ultrametric spaces generated by labeled rays},
  journal   = {Applied General Topology},
  year      = {2025},
  volume    = {26},
  number    = {1},
  pages     = {163--182}
}

@article{Kula2025,
 author = {Kula, M.},
 title = {Center of distances and {Bernstein} sets},
 journal = {Real Analysis Exchange},
 volume = {50},
 number = {1},
 pages = {193--198},
 year = {2025},
}

@Article{NowakowskiPrusWisniowski2026,
  author = {Nowakowski, P. and Prus-Wiśniowski, F.},
  title = {{A characterization of uniqueness of purely atomic finite measures with central Cantor set range}},
  journal = {ArXiv:2608.14395},
  year = {2026},
}

@book{SerreTrees1980,
  author    = {Serre, Jean-Pierre},
  title     = {Trees},
  translator= {Stillwell, John},
  publisher = {Springer-Verlag},
  address   = {Berlin--Heidelberg--New York},
  year      = {1980},
}

\end{document}